\documentclass[preprint]{elsarticle}

\usepackage{amssymb,amsmath,amsthm, color}
\usepackage{url}
\usepackage[utf8]{inputenc}
\usepackage{amsfonts,amssymb,amsmath,amsthm,bbm}
\usepackage{url}
\usepackage{enumerate}

\usepackage{amsmath,amsfonts,amsthm,url,color,amssymb, mathtools}
\usepackage{graphicx}
\usepackage[all]{xy}

\newtheorem{theorem}{Theorem}[section]
\newtheorem{lemma}[theorem]{Lemma}

\usepackage{amssymb}
\usepackage{enumerate}
\usepackage{graphicx}
\usepackage{algpseudocode, algorithm}
\usepackage{hyperref}

\newcommand{\C}{\mathbb{C}}

\newcommand{\F}{\mathbb F}
\newcommand{\Fq}{\mathbb{F}_q}
\newcommand{\Fqn}{\mathbb{F}_{q^n}}
\newcommand{\Sc}{\mathcal S_{C}^{*}}
\newcommand{\I }{\mathbbm {1}}
\newcommand{\doublespace}

\begin{document}

\begin{frontmatter}

\title{On primitive elements of finite fields avoiding affine hyperplanes}

\author[UFMG]{Arthur Fernandes}
\ead{arthurfpapa@gmail.com}

\author[UFMG]{Lucas Reis\corref{cor1}}
\ead{lucasreismat@gmail.com}

\address[UFMG]{Departamento de Matem\'{a}tica,
Universidade Federal de Minas Gerais,
UFMG,
Belo Horizonte MG (Brazil),
 30270-901}

\cortext[cor1]{Corresponding author}
 
\begin{abstract}
Let $n\ge 2$ be an integer and let $\Fq$ be the finite field with $q$ elements, where $q$ is a prime power. Given $\Fq$-affine hyperplanes $\mathcal A_1, \ldots, \mathcal A_n$ of $\Fqn$ in general position, we study the existence and distribution of primitive elements of $\Fqn$, avoiding each $\mathcal A_i$. We obtain both asymptotic and concrete results, relating to past works on digits over finite fields. 
\end{abstract}

\journal{Elsevier}
\end{frontmatter}




\section{Introduction}
Let $\F_q$ be the finite field with $q$ elements, where $q$ is a prime power. It is well known that the multiplicative group $\F_q^*$ is cyclic and any generator of such group is called {\em primitive}. For an integer $n\ge 2$, the $n$-degree extension  $\F_{q^n}$ of $\F_q$ can be regarded as an $\F_q$-vector space of dimension $n$. If $\mathcal B=\{b_1, \ldots, b_k\}$ is an $\F_q$-basis for $\F_{q^n}$, then every element $y\in \F_{q^n}$ is written uniquely as $$y=a_1b_1+\cdots+a_nb_n,\; a_i\in \F_q.$$ Inspired by works on the digital expansion of integers in a given basis~\cite{MR, MR2}, 
 Dartyge and Sarkozy~\cite{dig2} introduced the notion of digits over finite fields. In the previous setting, the elements $a_1, \ldots, a_n$ are the {\em digits} of $y$ with respect to the basis $\mathcal B$. In the past few years, many authors have explored the existence and number of special elements of finite fields (squares, perfect powers and primitive elements) with prescribed digits~\cite{dig2, dig2', R21, R21', dig4, dig5}. 
 
 In~\cite{dig2'}, the authors study the existence of special elements (squares and polynomial values with primitive arguments) with {\em missing digits}, i.e., elements 
 $y=\sum_{i=1}^na_ib_i$, where each digit $a_i$ is restricted to a subset $\mathcal D_i$ of $\F_p$ ($p$ is a prime). In particular, they proved that a square of this form always exists if each quantity $\#\mathcal D_i$ is (roughly) at least
 $\frac{(\sqrt{5}-1)p}{2}$ as $p\to +\infty$. This result is further improved in~\cite{sh}, where milder conditions are imposed on the quantities   $\#\mathcal D_i$. The main idea employed in the proof of those results is to provide a nontrivial bound for the sum
 \begin{equation}\label{eq:intro}s\left(\mathcal S, \chi\right):=\left|\sum_{y\in \mathcal S}\chi(y)\right|,\end{equation}
where $\chi$ is a nontrivial multiplicative character of $\F_{q^n}$ and $\mathcal S$ corresponds to the elements of $\F_{q^n}$ whose digits are restricted to certain subsets of $\F_q$. By nontrivial we mean a bound $s\left(\mathcal S, \chi\right)= o(\#\mathcal S)$ as $\# \mathcal S\to +\infty$.

Motivated by the latter, this paper aims to address a concrete result on the existence of primitive elements with missing digits. More specifically, for elements $c_1, \ldots, c_n\in \F_q$ and an $\F_q$-basis $\mathcal B$ of $\F_{q^n}$, we study the existence of primitive elements $y\in \F_{q^n}$ such that, for each $1\le i\le n$, the corresponding digit $a_i$ satisfies $a_i\ne c_i$. This is equivalent to consider the case where $\# \mathcal D=1$. 

Our approach also relies on providing a non trivial bound for a sum like the one in Eq.~\eqref{eq:intro} and, in particular, we asymptotically recover the bound in~\cite{dig2'} for $\# D=1$. However, our approach is quite simpler in this case, where we only use a known bound for multiplicative character sums over affine spaces. The latter allows us to obtain a more precise estimate for specific values of $q$ and $n$, which is quite helpful in the proof of Theorem~\ref{thm:main}. Moreover, we obtain a fairly more general result. For elements $c_1, \ldots, c_n\in \F_q$ and an $\F_q$-basis $\mathcal B=\{b_1, \ldots, b_n\}$ of $\F_{q^n}$, we observe that each set 
$$\mathcal A_i=\left\{\sum_{i=1}^na_ib_i\,|\, a_i=c_i\right\},$$
determines an $\F_q$-affine hyperplane of $\F_{q^n}$. Moreover the set $C=\{\mathcal A_1, \ldots, \mathcal A_n\}$ comprises $\F_q$-affine hyperplanes in {\em general position}, i.e., for each $1\le k\le n$, the intersection of any $k$ distinct elements of $C$ is an $\F_q$-affine space of dimension $n-k$. In this context, our main result can be stated as follows. 

\begin{theorem}\label{thm:main}
Let $n\ge 2$ be a positive integer and let $C=\{\mathcal A_1, \ldots, \mathcal A_n\}$ be a set of $\Fq$-affine hyperplanes of $\F_{q^n}$ in general position. Then the set $\mathcal S_{C}^*=\cup_{i=1}^n\F_{q^n}\setminus \mathcal A_i$ contains a primitive element of $\F_{q^n}$ provided that one of the following holds:

\begin{enumerate}[(i)]
    \item $q\ge 16$;
    \item $q=13$ and $n\ne 4$;
    \item $q=11$ and $n\ne 4, 6, 12$;
    \item $q=7, 8, 9$ and $n$ is large enough.
\end{enumerate}
\end{theorem}

In particular for $q\ge 16$, $n\ge 2$ and $c\in \F_q$, there exist primitive elements in $\F_{q^n}$ whose corresponding digits (in a given $\F_q$-basis) are all distinct from $c$.

Here goes the structure of the paper. In Section 2 we provide some tools that are used along the way. In Section 3 we estimate the number of primitive elements avoiding $\F_q$-affine hyperplanes, culminating in the proof of Theorem~\ref{thm:main}.
In Section 4 we comment on the possible extensions of Theorem~\ref{thm:main}.

\section{Preliminaries}
In this section we provide background material that is used throughout the paper.




\subsection{Characters}
A multiplicative character $\chi$ of a finite field $\F$ is a group homomorphism from $\F^*$ to $\C^*$. In other words,  $\chi(ab)=\chi(a)\cdot\chi(b)$ for every $a,b\in\F^*$. We usually extend $\chi$ to $0\in\F$ by setting $\chi(0)=0$.
It is known that the set of multiplicative characters of a finite field $\F$ is a cyclic group of order $\#\F-1$. In particular, for each divisor $d$ of $\# \F-1$, there exist $\varphi(d)$ multiplicative characters of $\F$ of order $d$.
The multiplicative character $\chi_0$ with $\chi_0(a)=1$ for every $a\in \F^*$ is the trivial multiplicative character; this is the unique multiplicative character of $\F$ of order $d=1$.

The following lemma, due to Vinogradov, provides a character sum formula for the characteristic function of the set of primitive elements in a finite field. Its proof is simple so we omit details.

\begin{lemma}[\cite{LN}, Exercise 5.14]\label{lem:func}\label{lem:vino}
 Let $\mu$ be the Möbius function, let $\varphi$ be the Euler Totient function and, for each divisor $d$ of $q^n-1$, let  $\Lambda(d)$ be the set of the multiplicative characters of $\Fqn$ of order $d$. The characteristic function $\I_\mathcal{P}$ for the set of primitive elements in $\Fqn$ can be expressed by $$\I_\mathcal{P}(\omega)=\dfrac{\varphi(q^n-1)}{q^n-1}\sum_{d|q^n-1}\dfrac{\mu(d)}{\varphi(d)}\sum_{\chi\in\Lambda(d)}\chi(\omega).$$
\end{lemma}

\subsection{Inequalities}
Here we provide some bounds that are useful along the way. 

\begin{lemma}\label{lem:tech-2}
For a positive integer $t$, let  $W(t)$ denote the number of squarefree divisors of $t$. For $t\ge 3$, the following hold:

\begin{enumerate}[1.]
    \item $W(t)< 4.9 \cdot t^{1/4}$;
    \item $W(t)<4514.7\cdot  t^{1/8}$;
    \item $W(t-1)<t^{\frac{0.96}{\log\log t}}$.
\end{enumerate}
\end{lemma}

\begin{proof}
Items 1 and 2 follow by Lemma 4.1 of~\cite{KR}, while item 3 follows by inequality (4.1) in~\cite{cohen-trud}.
\end{proof}


The following lemma provides a general bound for multiplicative character sums over affine spaces in finite fields. An interesting proof of this bound is essentially given in Corollary 3.5 of~\cite{R21}, with the help of Weil's bound. Here we provide a more direct proof, using a result from~\cite{GS}.

\begin{lemma}\label{lem:bound-af}
Let $\mathcal{A}\subseteq\Fqn$ be an  $\Fq$-affine space of dimension $t\ge 0$ and let $\chi$ be a non-trivial multiplicative character of $\Fqn$. We have that
    \begin{align*}
        s\left(\mathcal{A},\chi\right)\leq q^{\min\{t, \frac{n}{2}\}},
    \end{align*}
where $s\left(\mathcal{A},\chi\right)$ is as in \eqref{eq:intro}.
\end{lemma}

\begin{proof}
We have the trivial bound $s(\mathcal A, \chi)\le \# \mathcal A= q^t$, so it suffices to prove that $s(\mathcal A, \chi)\le q^{n/2}$. Let $\mathcal A=u+V$, where $V\subseteq \F_{q^n}$ is an $\F_q$-vector space of dimension $t$. Corollary~2 of~\cite{GS} implies that $$s\left(X+Y, \chi\right)\le (\#X\cdot \#Y\cdot q^n)^{1/2},$$ for any sets $X, Y\subseteq \F_{q^n}$. Taking $X=\mathcal A$ and $Y=V$, we observe that $s\left(\mathcal A+V, \chi\right)=\#V \cdot s\left(\mathcal A, \chi\right)$ and so 
$$\# V\cdot s\left(\mathcal A, \chi\right)\le \# V\cdot q^{n/2},$$
from where the result follows. 
\end{proof}

 \section{Primitive elements avoiding affine hyperplanes}
The following theorem provides a nice application of Lemma~\ref{lem:bound-af}.  
 
   \begin{theorem}\label{thm:main-bound}  Let $C=\{\mathcal A_1, \ldots, \mathcal A_n\}$ be a set of $\F_q$-affine hyperplanes of $\Fqn$ in general position and let $\chi$ be a non trivial multiplicative character over $\Fqn$. If $\Sc=\cap_{i=1}^n\Fqn\setminus \mathcal A_i$ and $\delta(q, n)=\displaystyle\sum_{i=0}^{n-1}\dbinom{n}{i}q^{\min\{i, \frac{n}{2}\}}$, then 
   
     $$
   s\left(\Sc,\chi\right)\le \delta(q, n) \le (2^n-1)q^{n/2}.$$
   
   \end{theorem}
   \begin{proof}
   For each set $X\subseteq \F_{q^n}$, let $\I_{X}$ be the characteristic function of the set $X$. Since $\Sc=\cap_{i=1}^n\Fqn\setminus \mathcal A_i$, we obtain that  $$\I_{\Sc}(\omega)=\prod_{i=1}^n (1-\I_{\mathcal{A}_i}(\omega))$$
   For each nonempty set $J\subseteq [1, n]:=\{1,\ldots,n\}$, we set $\mathcal{A}_J=\cap_{j\in J}\mathcal{A}_j$. In particular, we have the following equalities
    \begin{align*}
    s\left(\Sc,\chi\right) &= \left|\sum_{\omega\in\Fqn}\chi({\omega})\cdot\I_{\Sc}(\omega)\right|\\
    &=\left|\sum_{\omega\in\Fqn}\chi({\omega})\cdot\prod_{i=1}^n(1-\I_{\mathcal{A}_i}(\omega))\right|\\
    &=\left|\sum_{\omega\in\Fqn}\chi({\omega})\cdot\left(1+\sum_{i=1}^n\sum_{J\subseteq [1, n] \atop |J|=i}(-1)^{i}\I_{ \mathcal{A}_J}(\omega)\right)\right|\\
    &=\left|\sum_{\omega\in\Fqn}\chi({\omega})+\sum_{\omega\in\Fqn}\chi({\omega})\left(\sum_{i=1}^n\sum_{J\subseteq [1, n] \atop |J|=i}(-1)^{i}\I_{ \mathcal{A}_J}(\omega)\right)\right|.\\
    \end{align*}
    Since $\chi$ is non trivial,  $\displaystyle\sum_{\omega\in\Fqn}\chi({\omega})=0$ by orthogonality. In particular, we have that 
    $$
    s\left(\Sc,\chi\right) = \left|\sum_{i=1}^n\sum_{J\subseteq [1, n] \atop |J|=i}(-1)^{i}\left(\sum_{\omega\in\Fqn}\chi({\omega})\I_{ \mathcal{A}_J}(\omega)\right)\right|
    = \left|\sum_{i=1}^n\sum_{J\subseteq [1, n] \atop |J|=i}(-1)^{i}\left(\sum_{\omega\in \mathcal{A}_J}\chi({\omega})\right)\right|.
    $$
    From hypothesis, each $\mathcal A_{J}$ is an $\F_q$-affine space of dimension $n-|J|$. In particular, Lemma~\ref{lem:bound-af} and the Triangle Inequality imply that
    $$s\left(\Sc,\chi\right)\leq \sum_{i=1}^n\sum_{J\subseteq [1, n] \atop |J|=i}s\left(\mathcal{A}_J,\chi\right)\\
    \leq \sum_{i=1}^n\sum_{J\subseteq [1, n] \atop |J|=i}q^{\min\{n-i, \frac{n}{2}\}}\\
    =\sum_{i=0}^{n-1}\binom{n}{i}q^{\min\{i, \frac{n}{2}\}}.$$
   \end{proof}
   
   We obtain the following result.
\begin{theorem}\label{thm:prim}
Let $C=\{\mathcal A_1, \ldots, \mathcal A_n\}$ be a collection of $\F_q$-affine hyperplanes in general position in $\Fqn$ and let  $\mathcal{P(\Sc)}$ be the number of primitive elements in $\Sc=\cap_{i=1}^n\Fqn\setminus \mathcal A_i$. Then $$\dfrac{q^n-1}{\varphi(q^n-1)}\mathcal{P}(\Sc)>(q-1)^n-\delta(q, n)\cdot W(q^n-1),$$ where $\delta(q, n)$ is as in Theorem~\ref{thm:main-bound}.
\end{theorem}
\begin{proof}
From Lemma~\ref{lem:vino}, we obtain that
    $$\mathcal{P}(\Sc)=\sum_{w\in \Sc}\I_{\mathcal P}(w)=\sum_{\omega\in\Sc}\dfrac{\varphi(q^n-1)}{q^n-1}\sum_{d|q^n-1}\dfrac{\mu(d)}{\varphi(d)}\sum_{\chi\in\Lambda(d)}\chi(\omega).$$
    Therefore, 
    \begin{align*}
    \dfrac{q^n-1}{\varphi(q^n-1)}\mathcal{P}(\Sc)&=\sum_{\omega\in\Sc}\sum_{d|q^n-1}\dfrac{\mu(d)}{\varphi(d)}\sum_{\chi\in\Lambda(d)}\chi(\omega)\\
    &=\sum_{\omega\in\Sc}\chi_0(\omega)+\sum_{d\mid q^n-1 \atop d\neq 1}\dfrac{\mu(d)}{\varphi(d)}\sum_{\chi\in\Lambda(d)}\sum_{\omega\in\Sc}\chi(\omega).
    \end{align*}
    Since the elements of $C$ are in general position, a simple inclusion-exclusion argument implies that $S_C^*$ has cardinality $(q-1)^n$.  In particular, $$\displaystyle\sum_{\omega\in\Sc}\chi_0(\omega)=\sum_{\omega \in \Sc\atop \omega \ne 0}1\geq(q-1)^n-1.$$ 
    Therefore, from Theorem~\ref{thm:main-bound} and the Triangle Inequality we obtain that
    \begin{align*}
     \dfrac{q^n-1}{\varphi(q^n-1)}\mathcal{P}(\Sc)&\geq (q-1)^n-1-\left|\sum_{d\mid q^n-1 \atop d\neq 1}\dfrac{\mu(d)}{\varphi(d)}\sum_{\chi\in\Lambda(d)}\sum_{\omega\in\Sc}\chi(\omega)\right|\\
     &\geq(q-1)^n-1-\sum_{d\mid q^n-1\atop d\neq 1,\mu(d)\neq0}\dfrac{|\mu(d)|}{\varphi(d)}\sum_{\chi\in\Lambda(d)}\left|\sum_{\omega\in\Sc}\chi(\omega)\right|\\
     &\geq(q-1)^n-1-\sum_{d\mid q^n-1 \atop d\neq 1,\mu(d)\neq0}\dfrac{1}{\varphi(d)}\sum_{\chi\in\Lambda(d)}\delta(q, n)\\
     &=(q-1)^n-1-\delta(q, n)\sum_{d\mid q^n-1 \atop d\neq 1,\mu(d)\neq0}1\\
     &>(q-1)^n-\delta(q, n)\cdot W(q^n-1).
\end{align*}
\end{proof}


\subsection{Concrete results: proof of Theorem~\ref{thm:main}}
In this section we provide the proof of Theorem~\ref{thm:main}. Here and throughout, $C=\{\mathcal A_1, \ldots, \mathcal A_n\}$ is a collection of $\F_q$-affine hyperplanes of $\F_{q^n}$ in general position and $\Sc=\cup_{i=1}^n\F_{q^n}\setminus \mathcal A_i$. Theorem~\ref{thm:prim} entails that $\Sc$ contains a primitive element whenever the following inequality holds
\begin{equation}\label{eq:b}W(q^n-1)\le \left(\frac{q-1}{2\sqrt{q}}\right)^n.\end{equation}
Since $W(t)\ge 2$ for every $t\ge 2$, we have the trivial restriction $q-1>2\sqrt{q}$, i.e., $q\ge 7$. Moreover, if $q\ge 7$ is fixed, item 3 of Lemma~\ref{lem:tech-2} entails that Eq.~\eqref{eq:b} holds if $n$ is large enough. The latter proves item (iv) of Theorem~\ref{thm:main}.

From now and on, we assume that $q\ge 7$. If we set $F_1(t)=4.9 \cdot t^{1/4}, F_2(t)=4514.7 \cdot t^{1/8}$ and $F_3(t)=t^{\frac{0.96}{\log\log t}}$, Lemma~\ref{lem:tech-2} entails that 
$W(q^n-1)\le \min\limits_{1\le i\le 3} F_i(q^n)$.
So it suffices to have the following inequality
\begin{equation}\label{eq:c}\min _{1\le i\le 3} F_i(q^n)^{1/n}\le \frac{q-1}{2\sqrt{q}}.\end{equation}
We observe that the functions $F_i(q^n)^{1/n}$ are decreasing on $n$ in the range $q\ge 7$ and $n\ge 2$. By a direct computation we obtain Table~\ref{tab:my_label}, that displays some ranges where Eq.~\eqref{eq:c} holds. 

\begin{table}[h!]
    \centering
    \begin{tabular}{|c|c|c|}\hline 
         $q$ & $n$ & $F_i$ \\ \hline 
         $\ge 389$ & $\ge 2$  & $F_1$\\ \hline
         $\ge 76$ & $\ge 4$ & $F_3$\\ \hline 
         $\ge 16$ & $\ge 25$& $F_3$\\ \hline 
         $\ge 13$ & $\ge 45$ &  $F_2$\\ \hline
         $\ge 11$ & $\ge 76$ & $F_2$\\ \hline
    \end{tabular}
    \caption{Pairs $(q, n)$ satisfying Eq.~\eqref{eq:c} and the function $F_i$ employed.}
    \label{tab:my_label}
\end{table}

For the finite set $\mathcal X$ of pairs in the range $q\ge 11$ and $n\ge 2$ that are not included in Table~\ref{tab:my_label}, we proceed to direct computations:
\begin{enumerate}[(a)]
    \item  Theorem~\ref{thm:prim} also entails that $\Sc$ contains a primitive element whenever the following inequality holds $$(q-1)^n>\delta(q, n) W(q^n-1).$$

\item It is known that there exist $\varphi(q^n-1)$ primitive elements in $\F_{q^n}^*$. Since the set $\Sc$ contains $(q-1)^n$ elements, such a set contains a primitive element whenever the following inequality holds $$(q-1)^n+\varphi(q^n-1)>q^n.$$
\end{enumerate}
Using a SageMATH program, we directly verify that, with the exception of the pairs
$$(q, n)=(13, 4); (11, 4); (11, 6); (11, 12),$$
the elements of $\mathcal X$ satisfy the inequality in one of the items (a) or (b) above. This completes the proof of Theorem~\ref{thm:main}.

\section{Conclusions}
Motivated by works on special elements of finite fields with restrictions on their digits, this paper discussed the existence of primitive elements of finite fields avoiding affine hyperplanes in general position. We obtained a complete result for extensions $\F_{q^n}$ with $q\ge 16$ and $n\ge 2$. Moreover, the cases $q=11, 13$ yield only $4$ possible exceptions and we also obtained an asymptotic result for $q=7, 8, 9$.

By using a sieving method that is traditional in this kind of problem (see~\cite{cohen-trud}), one can also check that the pair $(q, n)=(11, 12)$ is not a genuine exception in Theorem~\ref{thm:main}. It would be interesting to achieve concrete results for $q=7, 8, 9$ or even discuss asymptotic results for $q=3, 4, 5$ (the case $q=2$ is a trivial exception for every $n\ge 2$). In order to accomplish the latter, we believe that sharper bounds on multiplicative characters sums related to this problem must be obtained. 

\begin{center}{\bf Acknowledgments}\end{center}
We would like to thank Fabio Brochero for providing helpful suggestions in an earlier version of this work.



\end{document}